\documentclass[preprint,12pt]{elsarticle}
\usepackage[utf8]{inputenc}

\usepackage{fullpage}

\usepackage{xcolor}
\usepackage{amsmath}
\usepackage{amssymb}
\usepackage{latexsym}
\usepackage{amsthm}
\usepackage{mathrsfs}
\usepackage{multicol}
\usepackage{xfrac}

\usepackage{tikz}
\usetikzlibrary{arrows}

\tikzstyle{arc}=[->,shorten <=3pt, shorten >=3pt,
                 >=stealth, line width=1.1pt]
\tikzstyle{edge}=[shorten <=2pt, shorten >=2pt,
                  >=stealth, line width=1.1pt]
\tikzstyle{vertex}=[circle, fill=white, draw,
                    minimum size=7pt,
                    inner sep=0pt, outer sep=0pt]
\tikzstyle{bVertex}=[circle, fill=white, draw,
                     minimum size=7pt, fill=black,
                     inner sep=0pt, outer sep=0pt]

\usepackage{url}

\newtheorem{theorem}{Theorem}
\newtheorem{lemma}[theorem]{Lemma}
\newtheorem{corollary}[theorem]{Corollary}
\newtheorem{proposition}[theorem]{Proposition}
\newtheorem{observation}[theorem]{Observation}

\newtheorem{question}[theorem]{Question}
\newtheorem{claim}{Claim}

\usepackage{minitoc}

\usepackage{ifpdf}

\journal{Elsevier}

\begin{document}

\begin{frontmatter}



\title{{Duality pairs and homomorphisms to oriented and unoriented cycles\tnoteref{t1}}}

\tnotetext[t1]{This research was supported by SEP-CONACYT grant A1-S-8397.}

\author[FC]{Santiago Guzm\'an-Pro}
\ead{sanguzpro@ciencias.unam.mx}

\author[FC]{C\'esar Hern\'andez-Cruz\corref{cor1}}
\ead{chc@ciencias.unam.mx}

\address[FC]{Facultad de Ciencias\\
Universidad Nacional Aut\'onoma de M\'exico\\
Av. Universidad 3000, Circuito Exterior S/N\\
C.P. 04510, Ciudad Universitaria, CDMX, M\'exico}

\cortext[cor1]{Corresponding author}

\begin{abstract}
In the homomorphism order of digraphs, a duality
pair is an ordered pair of digraphs $(G,H)$ such that for
any digraph, $D$, $G\to D$ if and only if $D\not \to H$.
The directed path on $k+1$ vertices together with the
transitive tournament on $k$ vertices is a classic
example of a duality pair. This relation between paths
and tournaments implies that a graph is $k$-colourable
if and only if it admits an orientation with no
directed path on more than $k$-vertices.

In this work, for every undirected cycle $C$ we find an
orientation $C_D$ and an oriented path $P_C$, such that
$(P_C,C_D)$ is a duality pair.  As a consequence we obtain
that there is a finite set, $F_C$, such that an undirected
graph is homomorphic to $C$, if and only if it admits an
$F_C$-free orientation. As a byproduct of the proposed
duality pairs, we show that if $T$ is a tree of height
at most $3$, one can choose a dual of $T$ of linear size
with respect to the size of $T$.
\end{abstract}


\end{frontmatter}



\section{Introduction}
\label{sec:Introduction}

Our main result can be considered in three different
contexts.   We now present a brief introduction to
each of them.

The Roy-Gallai-Hasse-Vitaver Theorem \cite{gallaiPCT,
hasseIMN28, royIRO1, vitaverDAN147} states that a graph
is $k$-colourable if and only if it admits an orientation
with no directed path on more than $k$ vertices. This
result is a consequence of the fact that a digraph $D$ is
homomorphic to the transitive tournament on $k$ vertices,
$TT_k$, if and only if the directed path on $k+1$ vertices,
$\overrightarrow{P}_{k+1}$, is not homomorphic to $D$.
In terms of duality pairs, $(\overrightarrow{P}_{k+1},
TT_k)$ is a duality pair in the homomorphism order of
digraphs. In \cite{nesetrilJCTB80} Ne\v{s}et\v{r}il and
Tardif proved that if $(A,B)$ is a duality pair in the
homomorphism order of digraphs, then $A$ is an oriented
tree.  Moreover, for any oriented tree, $T$, there is a
digraph $D_T$ (the {\em dual of} $T$), such that
$(T,D_T)$ is a duality pair in the homomorphism order
of digraphs. Their result is actually more general,
dealing with relational structures, so, as other authors
have done, we consider a restriction for the context of
this work. In fact, in \cite{nesetrilEJC29}, the same
authors consider the problem restricted to digraphs and,
for a given oriented tree $T$, they construct a digraph
$D_T$ such that $(T,D_T)$ is a duality pair. Their
construction is simple, but of size exponential on $|V_T|$,
raising the following question, can one choose $D_T$ to
be of polynomial size with respect to $|V_T|$? For instance,
for the family of directed paths, one can choose $D_T$ to
be the corresponding dual transitive tournament, and thus
$D_T$ is of linear size when $T$ is a directed path.

Similar notions of duality have been studied also in the
context of digraph homomorphisms. In \cite{hellC13}, Hell
and Zhu defined the class of $B$-cycles as special
orientations of cycles, and showed that for a fixed
$B$-cycle, $C$, a digraph $D$ is not homomorphic to $C$,
if and only if there exists a path $P$ homomorphic to $D$,
which is not homomorphic to $C$.   They call this notion
of duality {\em path duality}.

For a set of oriented graphs $F$ the class of $F$-graphs
is the class of undirected graphs that admit an $F$-free
orientation. In \cite{skrienJGT6}, Skrien found a structural
characterization for the class of $F$-graphs when $F$ is a
set of oriented paths on $3$-vertices. Some of these are
proper interval graphs, proper circular-arc graphs and
comparability graphs. In \cite{guzmanArXiv}, Skrien's
study of $F$-graphs is extended to any set of oriented
graphs on $3$ vertices. Two of these classes are still
lacking a complete structural characterization; the
so-called perfectly orientable graphs \cite{skrienJGT6},
and the transitive-perfectly orientable graphs
\cite{guzmanArXiv}. In terms of $F$-graphs, the
Roy-Gallai-Vitaver-Hasse Theorem states that, when $F$
is the set of oriented graphs on $k+1$ vertices with a
hamiltonian directed path, the class of $F$-graphs is
the class of $k$-colourable graphs.  In this case, one
can assume that such an orientations is also acyclic.
The class of graphs that admit an acyclic $F$-free
orientation is the class of $F^*$-graphs \cite{skrienJGT6}.
Another example of such classes are chordal graphs: when
$F$ consists of the orientation of the path on $3$ vertices
such that one vertex has $2$ out-neighbours, the class of
$F^*$-graphs is the class of chordal graphs. This statement
follows from the fact that a graph is chordal if and only
if it admits a perfect elimination ordering
\cite{fulkersonPJM15}.

Even though we mainly deal with duality pairs in the
homomorphism order of digraphs,  the whole paper is
motivated by the study of characterizations of graph
classes as $F$-graphs, for a finite set $F$. For each
positive integer $n$, $n \ge 3$, we present a finite
set of oriented graphs $F_n$ such that $F_n$-graphs
are precisely $C_n$-colourable graphs, i.e., graphs
that admit a homomorphism to the $n$-cycle. In a way
similar to the Roy-Gallai-Vitaver-Hasse Theorem, we
use duality pairs as a tool to find such a set $F_n$.

From the viewpoint of oriented cycles and path dualities,
it turns out that our result yields another class of
cycles, in addition to the $B$-cycles studied in
\cite{hellC13}, having path duality. The class we propose,
$AC$-cycles, is somewhat more restrictive, but the result
can be strengthened: for any $AC$-cycle, $C$, we obtain an
oriented path $P_C$, such that a digraph $D$ is not
homomorphic to $C$, if and only if $P_C$ is homomorphic
to $D$.

The class of $AC$-cycles corresponds to the family of
duals, $D_P$, for oriented paths $P$ in a special set,
for the moment the set $\mathcal{Q}$. Moreover, these
$AC$-cycles are duals of linear size with respect to their
corresponding path $P$. In \cite{hellEJC12} Hell and
Ne\v{s}et\v{r}il showed that the core of any oriented
tree of height $3$ is a path in $\mathcal{Q}$. Hence,
we conclude that for any tree $T$ of height at most $3$,
one can choose a dual $D_T$ of linear size with respect
to the core of $T$, and thus of linear size with respecto
to $T$.

The rest of this work is structured as follows. In
Section~\ref{sec:PreliminaryResults}, we introduce
basic notation, concepts and results needed for
later developments. Our main result is stated and
proved in Section~\ref{sec:Main}. Finally, in
Section~\ref{sec:context} we consider the different
interpretations of our main result in the three
contexts introduced above.   Conclusions are
briefly presented in Section \ref{sec:Conclusion}.

\section{Preliminary results}\label{sec:PreliminaryResults}

When $G$ and $H$ are graphs, we write $G \to H$ to
denote that $G$ is homomorphic to $H$. When $x$ and
$y$ are vertices of a digraph $D$, we write $x \to
y$ to denote that $(x,y)$ is an arc of $D$. It
should always be clear from the context to which
interpretation of the symbol $\to$ we are referring
to. Nonetheless, when speaking of homomorphisms, we
will use capital letters for digraphs, and when dealing
with arcs in a digraph, we will use small-case letter
for vertices.

An {\em oriented path} $P$ is a sequence of distinct
vertices $(p_0, \dots, p_n)$ such that, for each
$i \in \{0, \dots, n-1\}$, either $p_i p_{i+1} \in
A_P$, or $p_{i+1} p_i \in A_P$ (but not both), and
$P$ has no more arcs. If $p_i \to p_{i+1}$ we say that
$(p_i, p_{i+1})$ is a {\em forward arc}; if $p_{i+1}
\to p_i$ the arc $(p_{i+1},p_i)$ is a {\em backward
arc}. The direction in which $P$ is traversed is
emphasized by saying that the {\em initial vertex of
$P$} is $p_0$ and the {\em terminal vertex of $P$} is
$p_n$. If all arcs in $P$ are forward (backward) arcs,
we say that $P$ is a {\em directed path}, with forward
(backward) direction and denote it by
$\overrightarrow{P}_{n+1}$ ($\overleftarrow{P}_{n+1}$),
where $n$ is the number or arcs of $P$. An oriented
path is {\em alternating} if every two successive arcs
are oppositely oriented. We denote by $A_n$, the
alternating path on $n$ vertices that begins with
a forward arc, if $n=1$, then $A_n$ denotes the
single vertex with no arcs. A {\em semi-walk} on a
digraph $D$, is a sequence $v_1 a_1 v_2 a_2 \dots
a_{n-1} v_n$, where $v_i \in V_D$ for $i \in \{ 1,
\dots, n\}$, and $a_i$ is an arc with endpoints
$v_i$ and $v_{i+1}$, for $i \in \{1, \dots, n-1\}$.
An arc $a_i$ in a semi-walk is a {\em forward arc}
if $a_i = (v_i,v_{i+1})$; otherwise, we say it is a
{\em backward arc}. A semi-walk is {\em closed}
if $v_1 = v_n$.  The {\em pattern} of the semi-walk
$v_1 a_1 v_2 a_2 \dots v_n$, is a sequence $l_1
\dots l_{n-1}$ of symbols in $\{\to, \leftarrow\}$,
where $l_i = \to$ if $a_i$ is a forward arc; $l_i =
\leftarrow$ otherwise.

An {\em oriented cycle} $C$ is an oriented graph
obtained by identifying the initial and terminal
vertex of an oriented path $P$. If all arcs have the
same direction, we speak of a {\em directed cycle},
and denote it by $\overrightarrow{C}_n$.

The {\em net length} $\ell(X)$ of an oriented path
or oriented cycle, $X$, is the number of forward arcs
minus de number of backward arcs in $X$. The following
statement if proved in \cite{haggkvistC8}, but we use the
restatement found in \cite{hellJDM8} for its simplicity.

\begin{theorem}\label{directedpath}\cite{hellJDM8}
For $n \ge 1$, an oriented graph $G$ is homomorphic
to $\overrightarrow{P}_n$, if and only if, every
oriented path homomorphic to $G$ has net length at
most $n$.
\end{theorem}

Theorem~\ref{directedpath} shows that directed paths
have path duality. Now we introduce another family
of oriented graphs, proposed by Hell and Zhu in
\cite{hellC13}, that have path duality. An oriented
path $P$ is {\em minimal} if it contains no proper
oriented $P'$ such that $\ell(P') = \ell(P)$. An
oriented cycle $C = (c_0, \dots, c_n, \dots, c_{m-1},
c_0)$ is a {\em $B$-cycle}, if $(c_0, \dots, c_n)$
is a forward directed path, and $(c_0, c_{m-1},
\dots, c_n)$ is a minimal oriented path of net
length $n-1$.  As mentioned in
Section~\ref{sec:Introduction}, $B$-cycles have path
duality.

\begin{theorem}\label{Bcycles}\cite{hellC13}
Let $C$ be a $B$-cycle. A digraph $D$ is homomorphic
to $C$ if and only if every oriented path homomorphic
to $D$ is also homomorphic to $C$.
\end{theorem}

A digraph $G$ is balanced if every oriented cycle in
$G$ has net length zero. A digraph on $n$ vertices is
balanced if and only if $D \to \overrightarrow{P}_{n-1}$
(see \cite{hell2004}). Since every directed cycle has
positive net length, every balanced digraph must be
acyclic, and thus there is at least one vertex with no
in-neighbours. Let $G$ be a connected balanced digraph
and $x \in V_G$ such that $d^-(x) = 0$. We define the
{\em level} of vertex $v \in V_G$ as the net length of
any oriented path from $x$ to $v$. The fact that the
level of every vertex is well-defined follows from the
choice of $G$, i.e. connected and balanced. The maximum
level of the vertices in $G$ is called the {\em height}
of $G$. For two digraphs $G$ and $H$, the {\em interval}
$[G,H]$ consists on all digraphs $M$ such that $G \to M
\to H$. The following statement is a useful and well-known
result about the homomorphism order of digraphs.

\begin{proposition}\cite{hellEJC12}\label{height3}
If $G$ is a balanced digraph of height $3$, then $G\in
[\overrightarrow{P}_3,\overrightarrow{P}_4]$.
\end{proposition}

Two oriented graphs $G$ and $H$ are {\em homomorphically
equivalent}, if and only if $G \to H$, and $H \to G$.
Thus, it follows that $G$ and $H$ are homomorphically
equivalent, if and only if, for any digraphs $L$ and $R$,
$L \to G$ if and only if $L \to H$, and, $G \to R$ if
and only if $H \to R$. An ordered pair of digraphs $(G,H)$
is a {\em duality pair}, if for any digraph $L$, $G \not
\to L$ if and only if $L \to H$. In this case, we say that
$H$ is a {\em dual} of $G$. It is not hard to notice that
if such a dual exists, then it is unique up to homomorphic
equivalence. The transitive tournament on $n$ vertices is
denoted by $TT_n$. A classical example of a family of
duality pairs, is given by the following theorem.

\begin{theorem}\label{transitiveT}\cite{bloomJGT11}
For $n \ge 2$, an oriented graph $G$ is homomorphic to
$TT_n$ if and only if $\overrightarrow{P}_{n+1}$ is not
homomorphic to $G$, i.e., for every $n \ge 2$,
$(\overrightarrow{P}_{n+1},TT_n)$ is duality pair.
\end{theorem}

We conclude this section with the following straightforward
observation that we will use more than once in this work.

\begin{observation}\label{dpaireq}
Let $G, H$ and $R, S$ be pairs of homomorphically
equivalent oriented graphs, then $(G,R)$ is a duality
pair if and only if $(H,S)$ is a duality pair.
\end{observation}

\section{Main results.}\label{sec:Main}

We first define the family of oriented paths for which
we will find a family of duals. For $n \ge 3$ we denote
by $Q_n$ the oriented path on $n$ vertices $(q_0, \dots,
q_{n-1})$ with the following properties: the first two
arcs are forward arcs, the suboriented path $(q_1, \dots,
q_{n-2})$ is an alternating path, and the two final arcs
have the same direction. Note that, by the first two
conditions, $(q_1, \dots, q_{n-k}) = A_{n-(k+1)}$
for $1 \le k \le (n-3)$, and $(q_{n-3}, q_{n-2}, q_{n-1})
= \overrightarrow{P}_3$ or $(q_{n-3}, q_{n-2}, q_{n-1}) =
\overleftarrow{P}_3$, depending on the parity of $n$.
This is illustrated in Figure~\ref{Fig:AB}. In particular,
$Q_3$ and $Q_4$ are the directed paths on $3$ and $4$
vertices respectively.

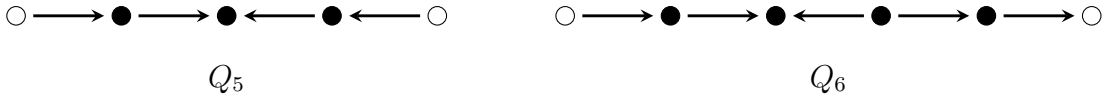
\begin{figure}[ht!]
\centering
\begin{tikzpicture}

\begin{scope}[xshift=-4cm, scale=0.7]
\node [vertex]  (1) at (-4,0){};
\node [bVertex] (2) at (-2,0){};
\node [bVertex] (3) at (0,0){};
\node [bVertex] (4) at (2,0){};
\node [vertex]  (5) at (4,0){};
\node[] at (0,-1.2){$Q_5$};
\node[] at (0,4.6){};

\draw[arc] (1) edge (2);
\draw[arc] (2) edge (3);
\draw[arc] (4) edge (3);
\draw[arc] (5) edge (4);
\end{scope}

\begin{scope}[xshift=4cm, scale=0.7]
\node [vertex]  (1) at (-5,0){};
\node [bVertex] (2) at (-3,0){};
\node [bVertex] (3) at (-1,0){};
\node [bVertex] (4) at (1,0){};
\node [bVertex] (5) at (3,0){};
\node [vertex]  (6) at (5,0){};
\node[] at (0,-1.2){$Q_6$};

\draw[arc] (1) edge (2);
\draw[arc] (2) edge (3);
\draw[arc] (4) edge (3);
\draw[arc] (4) edge (5);
\draw[arc] (5) edge (6);
\end{scope}

\end{tikzpicture}

\caption{The oriented paths $Q_{n+2}$ with the
vertices of their mid-section, $A_n$, coloured
black ($n \in \{ 3,4 \}$). In $Q_5$ the three
final vertices induce a directed path with all
arcs backward, while in $Q_6$ the three final
vertices induce a directed path with all arcs
forward.}
\label{Fig:AB}
\end{figure}

\begin{observation}\label{n-2}
For every integer $n$, $n \ge 5$, $Q_n$ is
homomorphic to $Q_{n-2}$. In particular, if
$n$ is even then $Q_n\to \overrightarrow{P}_4$,
and if $n$ is odd then $Q_n\to\overrightarrow{P}_3$.
\end{observation}

\begin{proof}
Let $n$ be an integer, $n \ge 5$, and let $Q_n=(q_0,
\dots,q_n)$. By identifying $q_3$ with $q_1$, and
$q_4$ with $q_2$, we obtain a homomorphism from $Q_n$
to $Q_{n-2}$.
\end{proof}

It is also straightforward to calculate the net
length of the oriented paths $Q_n$.

\begin{observation}\label{netL}
If $n$ is an  odd integer, $n \ge 3$, then $\ell(Q_n)
= 3$; if $n$ is an even integer, $n \ge 4$, then
$\ell(Q_n) = 4$.
\end{observation}

Thus, by Observations~\ref{n-2} and \ref{netL}, and
Theorem~\ref{directedpath} the following statement
holds.

\begin{lemma}\label{P3eq}
For every integer $n$, $n \ge 4$, $Q_n$ is
homomorphically equivalent to $\overrightarrow{P}_3$
if and only if $n$ is odd.
\end{lemma}

For $n \ge 3$, we denote by $AC_n$ the oriented cycle
obtained from identifying the initial and terminal
vertices of the alternating path $A_{n+1}$.  For this
work we will denote the vertices of $AC_n$ as $(a_0,
a_1,\dots,a_{n-1},a_0)$. Note that, if $n$ is even,
every two consecutive arcs have opposite direction,
and if $n$ is odd, every pair of consecutive arcs,
except for $a_{n-1}\to a_0\to a_1$, have opposite
direction. In Figure~\ref{Fig:AAC} we illustrate
$A_4$, $AC_4$, $A_5$ and $AC_5$.

\begin{figure}[ht!]
\centering
\begin{tikzpicture}

\begin{scope}[xshift=-4cm, yshift = 2cm,scale=0.7]
\node [vertex,label={$a_0$}] (1) at (-4,0){};
\node [vertex,label={$a_1$}] (2) at (-2,0){};
\node [vertex,label={$a_2$}] (3) at (0,0){};
\node [vertex,label={$a_3$}] (4) at (2,0){};
\node [vertex,label={$a_4$}] (5) at (4,0){};
\node[] at (0,-1.2){$A_5$};

\draw[arc] (1)  edge  (2);
\draw[arc] (3)  edge  (2);
\draw[arc] (3)  edge  (4);
\draw[arc] (5)  edge  (4);
\end{scope}

\begin{scope}[xshift=4cm, yshift = 2cm,scale=0.7]
\node [vertex,label={$a_0$}] (1) at (-5,0){};
\node [vertex,label={$a_1$}] (2) at (-3,0){};
\node [vertex,label={$a_2$}] (3) at (-1,0){};
\node [vertex,label={$a_3$}] (4) at (1,0){};
\node [vertex,label={$a_4$}] (5) at (3,0){};
\node [vertex,label={$a_5$}] (6) at (5,0){};
\node[] at (0,-1.2){$A_6$};

\draw[arc] (1) edge (2);
\draw[arc] (3) edge (2);
\draw[arc] (3) edge (4);
\draw[arc] (5) edge (4);
\draw[arc] (5) edge (6);
\end{scope}

\begin{scope}[xshift=-4cm,yshift = -2cm, scale=0.8]
\node [vertex,label=225:{$a_4$}] (1) at (-1,0){};
\node [vertex,label=135:{$a_0$}] (2) at (-1,2){};
\node [vertex,label=45:{$a_1$}]  (3) at (1,2){};
\node [vertex,label=315:{$a_2$}] (4) at (1,0){};
\node[] at (0,-1.2){$AC_4$};
\node[] at (7,0){};

\draw[arc]    (2)  edge  (1);
\draw[arc]    (2)  edge  (3);
\draw[arc]    (4)  edge  (3);
\draw[arc]    (4)  edge  (1);

\end{scope}

\begin{scope}[xshift=4cm, yshift = -2cm,scale=0.8]
\node [vertex,label=225:{$a_3$}] (1) at (-1.5,0){};
\node [vertex,label=315:{$a_2$}] (2) at (1.5,0){};
\node [vertex,label=45:{$a_1$}]  (3) at (1.5,2){};
\node [vertex,label={$a_0$}]     (4) at (0,2){};
\node [vertex,label=135:{$a_4$}] (5) at (-1.5,2){};
\node[] at (0,-1.2){${AC_5}$};

\draw[arc] (2) edge (1);
\draw[arc] (2) edge (3);
\draw[arc] (4) edge (3);
\draw[arc] (5) edge (4);
\draw[arc] (5) edge (1);

\end{scope}
\end{tikzpicture}
\caption{The oriented paths $A_{n+1}$ and oriented
cycles $AC_n$ for $n\in\{4,5\}$.}
\label{Fig:AAC}
\end{figure}
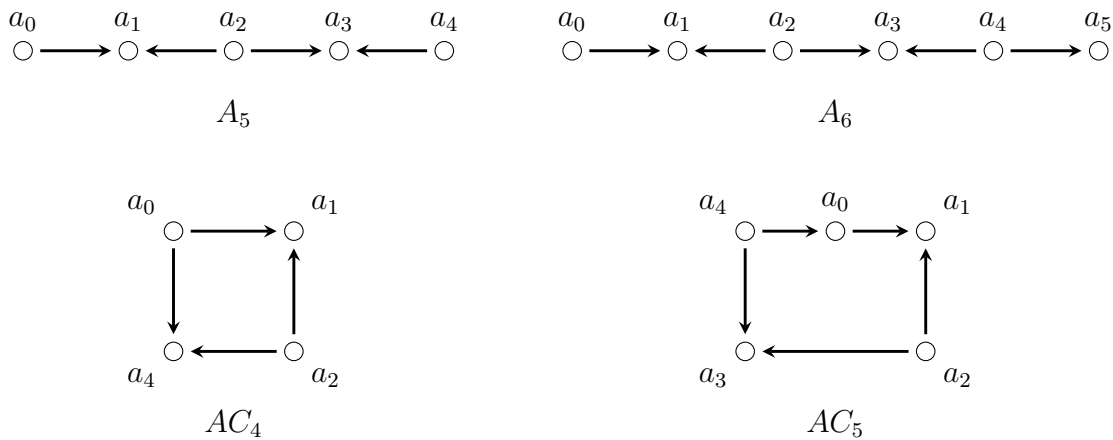

\begin{lemma}\label{P2eq}
For every integer $n$, $n \ge 4$, the cycle $AC_n$
is homomorphically equivalent to $\overrightarrow{P}_2$
if and only if $n$ is even.
\end{lemma}

\begin{proof}
Since $\overrightarrow{P}_2$ is an asymmetric arc,
then $\overrightarrow{P}_2$ is homomorphic to any
non-trivial oriented graph. If $n$ is even, every
two consecutive arcs of $AC_n$ have opposite direction,
thus, the largest directed path of $AC_n$ is
$\overrightarrow{P}_2$. But if $n$ is odd, there is a
copy of $\overrightarrow{P}_3$ contained in $AC_n$.
Therefore, by Theorem~\ref{transitiveT}, $AC_n$ is
homomorphic to $P_2$ if and only if $n$ is even.
\end{proof}

In order to avoid a very long proof for our main result, we attempt to
break it down in a reasonable amount of statements. We start with the
following one.

\begin{proposition}\label{QnotAC}
For every integer $n$, $n \ge 4$, the oriented path
$Q_{n+1}$ is not homomorphic to $AC_n$.
\end{proposition}

\begin{proof}
If $n$ is even, the result follows from Lemmas
\ref{P3eq} and \ref{P2eq}, Observation~\ref{dpaireq}
and Theorem~\ref{transitiveT}. Proceeding by
contradiction, suppose that there is an odd integer
$n \ge 5$ and a homomorphism $\varphi \colon Q_{n+1}
\to AC_n$. We first show that $\varphi$ is not a
surjective mapping. Recall that the only vertex in
$AC_n$ with in-degree and out-degree greater that
$0$ is $a_0$. Note that the only vertices in $Q_{n+1}
=(q_0\dots q_n)$ with in-degree and out-degree greater
than $0$, are $q_1$ and $q_{n-1}$. Thus, $\varphi(q_1)
= a_0 = \varphi(q_{n-1})$, and hence $\varphi(q_0) =
a_{n-1} = \varphi(q_{n-2})$ and $\varphi(q_2) = a_1 =
\varphi(q_n)$. Since $|V_{Q_n} - \{ q_0, q_1,
q_2, q_{n-2}, q_{n-1}, q_n\}| = n-5$, and $(n-5)+3 <
n = |V_{AC_n}|$, $\varphi$ is not surjective. Now we
observe that the existence of such a non-surjective
homomorphism leads to a contradiction. First, it is
not hard to verify that $AC_n-a_i$ is homomorphic to
$\overrightarrow{P}_3$ for any $i \in \{2, \dots, n-2
\}$. Since $\varphi$ is not surjective, and by
previous arguments $\{ a_{n-1}, a_0, a_1 \} \subseteq
\varphi[V_{Q_{n+1}}]$, then by composing homomorphisms,
$Q_{n+1}$ is homomorphic to $\overrightarrow{P}_3$.
Which contradicts the fact that $\ell(Q_{n+1}) = 4$
(Observation~\ref{netL}) and Theorem~\ref{directedpath}.
\end{proof}

Proposition \ref{QnotAC} implies that if $G$ is a
digraph and $Q_{n+1} \to G$ then $G \not \to AC_n$. To
prove the converse implication, for any connected
oriented graph $G$ and any odd integer $n$, $n \ge 5$,
we will construct an $n$-ordered cover of $V_G$, i.e.,
an ordered sequence of $n$ subsets of vertices that
cover $V_G$. We recursively define the {\em $n$-cyclic
cover} of a connected oriented graph $G$ as follows.

\begin{enumerate}
	\item If there is no vertex with in- and out-neighbours,
    let $A_0 = \{ v \in V_G \colon d^-(v) = 0\}$, $A_1 =
    \{ v \in V_G \colon d^+(v) = 0\}$. The $n$-cyclic cover
    of $G$ is $(A_0,A_1)$.

	\item Else, let $A_0 = \{ v \in V_G \colon d^-(v) > 0,
    d^+(v)>0 \}$, and let $m = \frac{n-1}{2}$.

	\item Let $A_1$ be the set of vertices in $V_G -
    A_0$, with an in-neighbour in $A_0$, $D_1$ the set of
    vertices in $V_G - A_0$, with an out-neighbour
    in $A_0$, and $C_1 = A_0 \cup A_1 \cup D_1$.

	\item For $i \in \{ 2, \dots, m-1 \}$, let $D_i$ be the
    set of vertices in $V_G - C_{i-1}$ with a
    neighbour in $D_{i-1}$, $A_i$ the set of vertices in
    $V_G - C_{i-1}$ with a neighbour in $A_{i-1}$,
    and $C_i = C_{i-1} \cup A_i \cup D_i$.

	\item If every vertex in $A_{m-1}$ has out-degree $0$,
    and every vertex in $D_{m-1}$ has in-degree $0$, let
    $D_m$ be the vertices in $V_G - C_{m-1}$ with
    no out-neigbours, and $A_m$ the vertices in $V_G
    - C_{m-1}$ with no in-neighbours.

	\item	Else, let $D_m$ be the vertices in $V_G -
    C_{m-1}$ with no in-neigbours, and $A_m$ the vertices
    in $V_G - C_{m-1}$ with no out-neighbours.

	\item The $n$-cyclic cover of $G$ is $(A_0, A_1, \dots,
    A_m, D_m, D_{m-1}, \dots, D_1)$.
\end{enumerate}

If the recursion finishes in the first step, i.e., $G$
has no vertices with in- and out-neighbours, we say that
the $n$-cyclic cover of $G$ is a {\em directed bipartition}
of $G$. The following simple properties of the $n$-cyclic
cover account for half the proof of our main result.

\begin{lemma}\label{cyclicCover}
Let $n$ be an odd integer, $n \ge 5$, and let $G$ be a
connected oriented graph with $n$-cyclic cover
$(A_0, A_1, \dots, A_m, D_m, D_{m-1}, \dots, D_1)$.
\begin{enumerate}
	\item For every $i \in \{ 1, \dots, m \}$, if $x \in
    D_i$ and $y \in A_i$, then $d^-(x) = d^+(y) = 0$
    if $i$ is even; $d^+(x) = d^-(y) = 0$ if $i$ is odd.

	\item The collection $(A_0, A_1, \dots, A_m, D_m,
    D_{m-1}, \dots, D_1)$ covers $V_G$ with pairwise
    disjoint sets.

	\item The sets of the $n$-cyclic cover of $G$ are
    independent if and only if $A_0$ is an independent set.

	\item The endpoints of every arc in $G$ either belong to
    consecutive sets in $(A_0, A_1, \dots, A_m, D_m, \dots,
    D_1, A_0)$, or belong to $D_i$ and $A_i$ for some
    $i \le \{ 1, \dots, m-1 \}$.
\end{enumerate}
\end{lemma}

\begin{proof}
The statements of this lemma are clear when the $n$-cyclic
cover is a directed bipartition, so will assume that the
vertices in $A_0$ have both in- and out-neighbours. Since
every vertex in $V_G - A_0$ has either empty out-neighbourhood
or empty in-neighbourhood, then, by definition of $D_1$
($A_1$), every vertex in $D_1$ ($A_1)$ has an
out(in)-neighbour in $A_0$, so it has an empty
in(out)-neighbourhood. For $i\in\{2,\dots,m\}$ the first
statement follows inductively.

To prove the second item, first note that $G$ is connected,
and thus, the sets $(A_0, A_1, \dots, A_m, D_m, D_{m-1},
\dots, D_1)$ cover $V_G$. By construction of these sets,
if $1 \le i < j \le m$, the following intersections are
empty: $A_i \cap D_j$, $A_i \cap A_j$, $D_i \cap D_j$, and
$D_i \cap A_j$. By the first statement, for every $i \in
\{ 1, \dots, m \}$, we have that $A_0 \cap D_i$, $A_0 \cap
A_i$ and $A_i \cap D_i$ are empty as well. Hence $(A_0, A_1,
\dots, A_m, D_1, \dots, D_m)$ is a cover of $V_G$ with
pairwise disjoint sets, i.e., a partition of $V_G$ with
possible empty sets.

For the third statement, note that for $i \in \{ 1, \dots,
m \}$, the existence of an arc within a set $A_i$ ($D_i$)
would imply that there is a vertex in $A_i$ ($D_i$) with
in- and out-degree at least one, which contradicts the fact
that $A_i, D_i \subseteq V_G - A_0$. Thus $A_i$ and $D_i$
are independent sets for $i \in \{ 1, \dots, m \}$. Hence,
the sets of the $n$-cyclic cover of $G$ are independent if
and only if $A_0$ is an independent set.

Finally, for a vertex $x \in V_G$, denote by $i(x)$ the
index of the partition class to which $x$ belongs to. Notice
that by the BFS style of constructing the elements of the
cover, if $(x,y)\in A_G$, then $|i(x)-i(y)|\le 1$. By the
first statement, for $i \in \{ 1, \dots, m-1 \}$ there are
no arcs between classes $D_i$ and $A_{i+1}$, nor between
$A_i$ and $D_{i+1}$. Therefore the last statement holds.
\end{proof}

For an oriented graph $G$ we denote by $Cyc(n,G)$ its
$n$-cyclic cover. If $Cyc(n,G)$ is not a directed
bipartition, we choose two functions, $l_n, r_n \colon
A_0 \to V_G$, such that $l_n(x)$ and $r_n(x)$ are in-
and out-neighbours of $x$, respectively. Similarly, we
choose $p_n \colon V_G - A_0 \to V_G$ any function such
that for $p_n(x)$ is a neighbour of $x$, for $i \in \{
2, \dots, m\}$ if $x \in A_i(D_i)$ then $p_n(x) \in
A_{i-1} (D_{i-1})$, and if $x \in A_1 \cup D_1$, then
$p_n(x) \in A_0$.

\begin{theorem}\label{iffCyc}
Let $n$ be an odd integer, $n \ge 5$, and $G$ an
oriented graph. Then, the following statements are
equivalent:
\begin{itemize}
	\item $G\to AC_n$,

	\item $Cyc(n,G)$ induces a homomorphism of $G$ to
    $AC_n$,

	\item for every even integer $l$, $4 \le l \le n+1$,
    and a semi-walk $W$ in $G$, $W$ does not follow
    the same pattern as $Q_l$, and

	\item $Q_{n+1} \not \to G$.
\end{itemize}
\end{theorem}

\begin{proof}
Clearly the second item implies the first one, and by
Proposition~\ref{QnotAC} the first item implies the
fourth one. Suppose that the negation of the third
item holds, i.e., there is a positive integer $l$, $4
\le l \le n+1$, and a semi-walk $W$ in $G$, such that
$W$ follows the same pattern as $Q_l$. Then there is a
homomorphism $\varphi \colon Q_l \to G$, so by
Observation~\ref{n-2} there is a homomorphism $\varphi
\colon Q_{n+1} \to G$. So by contrapositive, the fourth
item implies the third one.

Before showing that the third statement implies the
second one, we state the following claim.

\begin{claim}\label{clm}
If $A_0$ is an independent set, and there are no arcs
between $A_i$ and $D_i$ for any $i \in \{ 1, \dots, m-1
\}$, then $Cyc(n,G)$ induces a homomorphism of $G$
to $AC_n$.
\end{claim}

If $A_0$ is an independent set, then by
Lemma~\ref{cyclicCover}.3 every set in  $Cyc(n,G)$ is
independent. Moreover, if there are no arcs between
$D_i$ and $A_i$ for any $i \in \{ 1, \dots, m-1 \}$,
by Lemma~\ref{cyclicCover}.4 every arc in $G$ has
endpoints in consecutive sets of $(A_0, A_1, \dots,
A_m, D_m, \dots, D_1)$. Hence, the function $\varphi
\colon G \to AC_n$ defined by $\varphi(x) = a_i$ if
$x \in A_i$, and $\varphi(x) = a_{n-i}$ if $x \in
D_i$, is a homomorphism between the underlying graphs
of $G$ and $AC_n$. The fact that $\varphi$ also
preserves orientations follows from
Lemma~\ref{cyclicCover}.1.

Now, we proceed to prove that the third statement
implies the second one by contrapositive. By
Claim~\ref{clm} it suffices to show that if there is
an arc with either both endpoints in $A_0$, or one in
$D_i$ and the other in $A_i$ for some $i \in \{ 1,
\dots, m-1 \}$ then there is positive integer $l$, $4
\le l \le n+1$, and a semi-walk in $G$ that follows
the same pattern as $Q_l$. Suppose that there is an
arc $(x,y) \in A_G$ with $x, y \in A_0$, and let $W =
l_n(x) x y r_n(y)$. By the choice of $l_n(x)$ and
$r_n(y)$, and the fact that $x \to y$, we conclude
that $W$ follows the same pattern as $Q_4$.

Assume that there is an integer $k \in \{ 1,\dots,
m-1 \}$ and an arc with endpoints $a_k \in A_k$ and
$d_k \in D_k$. We construct two paths as follows;
let $W_a = a_{-1} a_0 a_1 \cdots a_k$, where $a_k =
a$, $a_{-1} = l_n(a_0)$ and for $i \in \{ 1, \dots,
k-1 \}$, $a_i = p_n(a_{i+1})$; and $W_d = d_{-1} d_0
d_1 \dots d_k$, where $d_k = d$, $d_{-1} = r_n(d_0)$
and for $i \in \{ 1,\dots, k-1 \}$, $d_i = p_n(d_{i+1})$.
Finally, let $W = a_{-1} W_a a_k d_k W_d d_{-1}$. First
note that the number of vertices (with possible
repetitions) in $W$ is $2k+4$. Clearly, $2k+4$ is
even, and since $k < \frac{n-1}{2}$ and $n+2$ is odd,
then $2k+4 \le n+1$. The fact that $W$ follows the same
pattern as $Q_{2k+4}$, is a consequence of
Lemma~\ref{cyclicCover}.1, the properties of $l_n, r_n$
and $p_n$, and the choice of $W_a$ and $W_d$, i.e.,
$a_{i-1} = p_n(a_i)$, $d_{i-1} = p_n(d_i)$, $a_{-1} =
l_n(a_0)$, and $d_{-1} = r_n(d_0)$. Therefore, if
$Cyc(n,G)$ does not induce a homomorphism of $G$ to
$AC_n$, there is a semi-walk in $G$ that follows the
pattern of $Q_l$ for some even integer $l$, $4 \le l
\le n+1$.
\end{proof}

It is straightforward to verify that if a digraph $D$
has a symmetric arc, then every oriented tree is
homomorphic to $D$. Also, if $G$ is an oriented graph,
then $D$ is not homomorphic to $G$. So if $T$ is an
oriented tree, and $D_T$ any of its duals, then $D_T$
is an oriented graph, $T \to D$, and $D \not \to D_T$.
For this reason, we state and prove the following theorem
for oriented graphs only, but clearly it also holds for
general digraphs.

\begin{theorem}\label{QAC-duality}
Let $n$ be an integer, $n \ge 4$, an oriented graph
$G$ is homomorphic to $AC_n$ if and only if $Q_n$ is
not homomorphic to $G$. In other words, the ordered
pair $(Q_{n+1},AC_n)$ is a duality pair.
\end{theorem}

\begin{proof}
If $n$ is even, by Lemma~\ref{P3eq}, $Q_{n+1}$ is
homomorphically equivalent to $\overrightarrow{P}_3$,
and by Lemma~\ref{P2eq} $AC_n$ is homomorphically
equivalent to $\overrightarrow{P}_2$. Thus, by
Observation~\ref{dpaireq}, if $n$ is even, $(Q_{n+1},
AC_n)$ is a duality pair if and only if
$(\overrightarrow{P}_3, \overrightarrow{P}_2)$ is a
duality pair. The later statement holds since
$\overrightarrow{P}_2 \cong TT_2$, and
$(\overrightarrow{P}_3,TT_2)$ is duality pair
(Theorem~\ref{transitiveT}). If $n = 3$, we conclude
by Theorem~\ref{transitiveT}. Finally, if  $n \ge 5$
is odd, we conclude by Theorem~\ref{iffCyc}.
\end{proof}

In Figure~\ref{DualityP} we exhibit two of the
duality pairs described in Theorem~\ref{QAC-duality}.

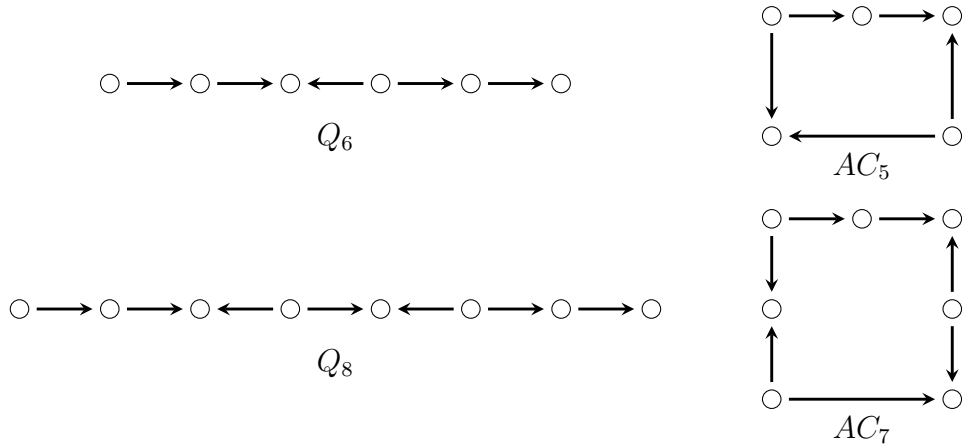
\begin{figure}[ht!]
\begin{center}
\begin{tikzpicture}

\begin{scope}[xshift=-3.5cm,yshift=1.5cm, scale=0.6]
\node [vertex] (0) at (-5,0){};
\node [vertex] (1) at (-3,0){};
\node [vertex] (2) at (-1,0){};
\node [vertex] (5) at (1,0){};
\node [vertex] (6) at (3,0){};
\node [vertex] (7) at (5,0){};
\node[] at (0,-1.2){$Q_6$};

\draw[arc] (0) edge (1);
\draw[arc] (1) edge (2);
\draw[arc] (5) edge (2);
\draw[arc] (5) edge (6);
\draw[arc] (6) edge (7);
\end{scope}

\begin{scope}[xshift=3.5cm, yshift=0.8cm,scale=0.8]
\node [vertex] (1) at (-1.5,0){};
\node [vertex] (2) at (1.5,0){};
\node [vertex] (3) at (1.5,2){};
\node [vertex] (4) at (0,2){};
\node [vertex] (5) at (-1.5,2){};
\node[] at (0,-0.5){${AC_5}$};

\draw[arc] (2) edge (1);
\draw[arc] (2) edge (3);
\draw[arc] (4) edge (3);
\draw[arc] (5) edge (4);
\draw[arc] (5) edge (1);

\end{scope}

\begin{scope}[xshift=-3.5cm, yshift=-1.5cm,scale=0.6]
\node [vertex] (0) at (-7,0){};
\node [vertex] (1) at (-5,0){};
\node [vertex] (2) at (-3,0){};
\node [vertex] (3) at (-1,0){};
\node [vertex] (4) at (1,0){};
\node [vertex] (5) at (3,0){};
\node [vertex] (6) at (5,0){};
\node [vertex] (7) at (7,0){};
\node[] at (0,-1.2){$Q_8$};

\draw[arc] (0) edge (1);
\draw[arc] (1) edge (2);
\draw[arc] (3) edge (2);
\draw[arc] (3) edge (4);
\draw[arc] (5) edge (4);
\draw[arc] (5) edge (6);
\draw[arc] (6) edge (7);
\end{scope}

\begin{scope}[xshift=3.5cm,yshift=-1.5cm, scale=0.8]
\node [vertex] (1) at (-1.5,1.5){};
\node [vertex] (2) at (-1.5,0){};
\node [vertex] (3) at (-1.5,-1.5){};
\node [vertex] (4) at (1.5,-1.5){};
\node [vertex] (5) at (1.5,0){};
\node [vertex] (6) at (1.5,1.5){};
\node [vertex] (7) at (0,1.5){};
\node[] at (0,-2){${AC_7}$};

\draw[arc] (1) edge (2);
\draw[arc] (3) edge (2);
\draw[arc] (3) edge (4);
\draw[arc] (5) edge (4);
\draw[arc] (5) edge (6);
\draw[arc] (7) edge (6);
\draw[arc] (1) edge (7);
\end{scope}
\end{tikzpicture}

\caption{Two duality pairs $(Q_6,AC_5)$ and $(Q_8,AC_7)$.}
\label{DualityP}
\end{center}
\end{figure}


\section{Implications}\label{sec:context}

We say that an oriented cycle $C$ is an $AC$-cycle
if $C \cong AC_n$ for some positive integer $n$.
The following result is a weaker version of
Theorem~\ref{QAC-duality}.

\begin{corollary}\label{ACcycles}
Let $C$ be an $AC$-cycle. A digraph $D$ is
homomorphic to $C$, if and only if every oriented
path homomorphic to $D$ is also homomorphic to $C$.
\end{corollary}

Thus, in terms of path dualities, we can extend
Theorem~\ref{Bcycles} with this corollary as follows.

\begin{theorem}
Any oriented cycle $C$ that is a $B$-cycle or an
$AC$-cycle, has path duality, i.e., a digraph $G$
is homomorphic to $C$, if and only if every path
homomorphic to $G$ is also homomorphic to $C$.
\end{theorem}

Recall that a digraph is a {\em core}, if and only
if it is not homomorphic to any proper subgraph.
The following statement is a well-known result
in homomorphism order of digraphs.

\begin{proposition}\label{Height3}\cite{hellEJC12}
Let $G$ be a digraph in $[\overrightarrow{P}_3,
\overrightarrow{P}_4]$, then $G$ is homomorphically
equivalent to $Q_n$ for some even integer $n \ge 4$.
Moreover for every even integer, $n \ge 4$, the path
$Q_n$ is a core.
\end{proposition}

Now, we give a partial answer to the problem of
determining if one can choose a dual, $D_T$, of
an oriented tree, $T$, of polynomial size with
respecto to $|V_T|$.

\begin{theorem} \label{tree-dual}
Let $T$ be an oriented tree of positive height at
most $3$, and $P_T$ its core. One can choose a dual
$D_T$ of $T$ of linear size with respect to $|V_{P_T}|$.
Since $|V_{P_T}| \le |V_T|$, then $D_T$ is of linear
size with respect to $|V_T|$.
\end{theorem}

\begin{proof}
If $T_1$ is a tree of height $1$, then
$\overrightarrow{P}_2$ is homomorphically
equivalent to $T_1$. When $T_2$ is a tree of height
$2$, then $T_2$ is homomorphically equivalent to
$\overrightarrow{P}_3$. So by Theorem~\ref{transitiveT},
$(T_1, TT_1)$ and $(T_2, TT_2)$ are duality pairs.

If $T_3$ is a tree of height $3$, then by
Propositions~\ref{height3} and \ref{Height3}, $T_3$
is homomorphically equivalent to a path $P$, and
$P \cong Q_{n+1}$ for an odd integer $n$, $n \ge 3$.
Thus, by Theorem~\ref{QAC-duality}, $(T_3, AC_n)$ is
a duality pair. For a tree $T$ in any of these cases,
the size of the chosen dual is linear with respecto
to the core of $T$.
\end{proof}

Finally, we connect our result to the study hereditary
graph properties characterized as the class of $F$-graphs
for a finite set $F$. For $n \ge 4$ an even integer,
denote by $F_n$, the set of surjective homomorphic
images of $Q_n$. Clearly, $F_n$ is a finite set since
the order of any oriented graph in $F_n$ is bounded by $n$.

\begin{theorem}
Let $G$ be a graph, $n \ge4$ an even integer, and $C$
the cycle on $n-1$ vertices. Then $G$ is
$C_{n-1}$-colourable if and only if $G$ is an $F_n$-graph.
That is, there is an orientation of $G$ with no induced
oriented graph in $F_n$.
\end{theorem}

By Observation~\ref{n-2}, the directed path on $4$
vertices belongs to $F_n$ for an even integer $n \ge 4$.
Thus, it is straightforward to notice that the directed
$3$- and $4$-cycles also belong to $F_n$. Thus, from the
previous corollary we obtain the following one.

\begin{corollary}
Let $G$ be a graph, $n \ge4$ an even integer, and $C$
the cycle on $n-1$ vertices. Then $G$ is $C_{n-1}$-colourable
if and only if $G$ is an $(F_n-\{ \overrightarrow{C}_3,
\overrightarrow{C}_4\})^\ast$-graph. That is, there is
an acyclic orientation of $G$ with no induced oriented graph
in $F_n-\{\overrightarrow{C}_3,\overrightarrow{C}_4\}$.
\end{corollary}

In particular, $F_6$ consists of the eight oriented graphs
depicted in Figure~\ref{Fig:F6}.

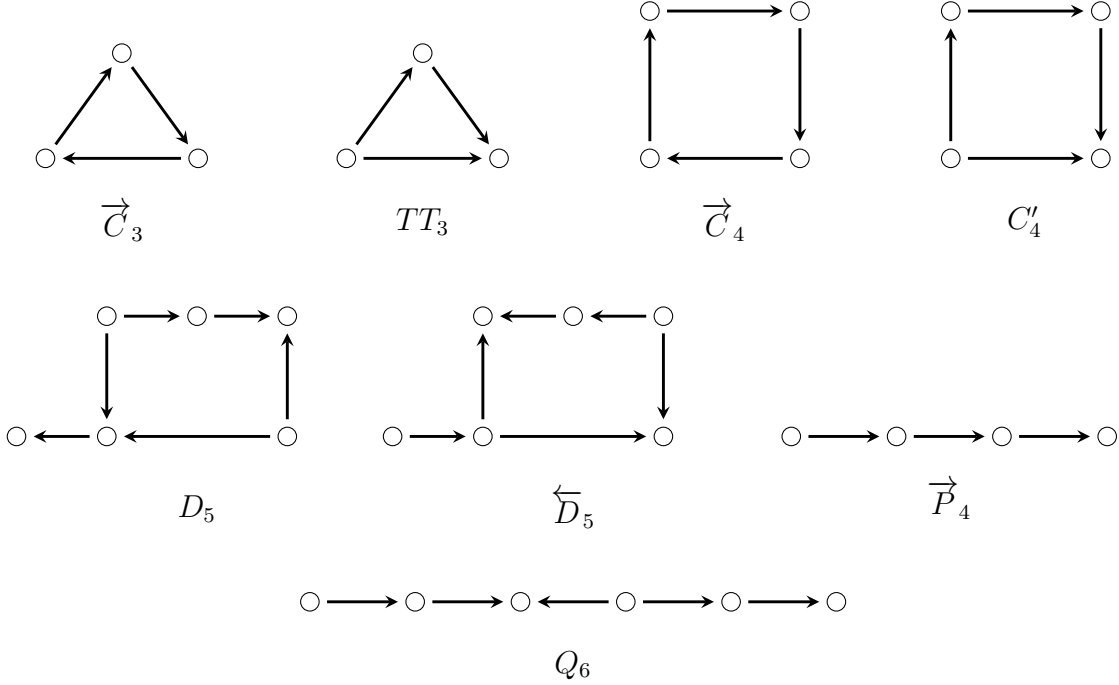
\begin{figure}[ht!]
\begin{center}

\begin{tikzpicture}

\begin{scope}[xshift=-6cm, yshift = 3.7cm, scale=0.7]
\node [vertex] (1) at (-1.45,0){};
\node [vertex] (2) at (0,2){};
\node [vertex] (3) at (1.45,0){};
\node[] at (0,-1.2){$\overrightarrow{C}_3$};

\draw[arc]  (1) edge (2);
\draw[arc]  (2) edge (3);
\draw[arc]  (3) edge (1);
\end{scope}

\begin{scope}[xshift=-2cm, yshift = 3.7cm, scale=0.7]
\node [vertex] (1) at (-1.45,0){};
\node [vertex] (2) at (0,2){};
\node [vertex] (3) at (1.45,0){};
\node[] at (0,-1.2){$TT_3$};

\draw[arc]  (1) edge (2);
\draw[arc]  (2) edge (3);
\draw[arc]  (1) edge (3);
\end{scope}

\begin{scope}[xshift=5cm, yshift=0cm, scale=0.7]
\node [vertex] (1) at (-3,0){};
\node [vertex] (2) at (-1,0){};
\node [vertex] (3) at (1,0){};
\node [vertex] (4) at (3,0){};
\node[] at (0,-1.2){$\overrightarrow{P}_4$};

\draw[arc]  (1) edge (2);
\draw[arc]  (2) edge (3);
\draw[arc]  (3) edge (4);

\end{scope}

\begin{scope}[xshift=2cm, yshift=3.7cm, scale=0.7]
\node [vertex] (1) at (-1.4,0){};
\node [vertex] (2) at (-1.4,2.8){};
\node [vertex] (3) at (1.45,2.8){};
\node [vertex] (4) at (1.45,0){};
\node[] at (0,-1.2){$\overrightarrow{C}_4$};

\draw[arc]  (1) edge (2);
\draw[arc]  (2) edge (3);
\draw[arc]  (3) edge (4);
\draw[arc]  (4) edge (1);
\end{scope}

\begin{scope}[xshift=6cm, yshift=3.7cm, scale=0.7]
\node [vertex] (1) at (-1.4,0){};
\node [vertex] (2) at (-1.4,2.8){};
\node [vertex] (3) at (1.45,2.8){};
\node [vertex] (4) at (1.45,0){};
\node[] at (0,-1.2){$C'_4$};

\draw[arc]  (1) edge (2);
\draw[arc]  (2) edge (3);
\draw[arc]  (3) edge (4);
\draw[arc]  (1) edge (4);
\end{scope}

\begin{scope}[xshift=0cm, yshift = -2.2cm,scale=0.7]
\node [vertex] (1) at (-5,0){};
\node [vertex] (2) at (-3,0){};
\node [vertex] (3) at (-1,0){};
\node [vertex] (4) at (1,0){};
\node [vertex] (5) at (3,0){};
\node [vertex] (6) at (5,0){};
\node[] at (0,-1.2){$Q_6$};

\draw[arc]  (1) edge (2);
\draw[arc]  (2) edge (3);
\draw[arc]  (4) edge (3);
\draw[arc]  (4) edge (5);
\draw[arc]  (5) edge (6);

\end{scope}

\begin{scope}[xshift=-5cm, yshift = 0cm, scale=0.8]
\node [vertex] (1) at (-1.5,0){};
\node [vertex] (2) at (1.5,0){};
\node [vertex] (3) at (1.5,2){};
\node [vertex] (4) at (0,2){};
\node [vertex] (5) at (-1.5,2){};
\node [vertex] (6) at (-3,0){};
\node[] at (0,-1.2){${D_5}$};

\draw[arc]  (2) edge (1);
\draw[arc]  (2) edge (3);
\draw[arc]  (4) edge (3);
\draw[arc]  (5) edge (4);
\draw[arc]  (5) edge (1);
\draw[arc]  (1) edge (6);

\end{scope}

\begin{scope}[xshift=0cm, yshift = 0cm, scale=0.8]
\node [vertex] (1) at (-1.5,0){};
\node [vertex] (2) at (1.5,0){};
\node [vertex] (3) at (1.5,2){};
\node [vertex] (4) at (0,2){};
\node [vertex] (5) at (-1.5,2){};
\node [vertex] (6) at (-3,0){};
\node[] at (0,-1.2){$\overleftarrow{D}_5$};

\draw[arc]  (1) edge (2);
\draw[arc]  (3) edge (2);
\draw[arc]  (3) edge (4);
\draw[arc]  (4) edge (5);
\draw[arc]  (1) edge (5);
\draw[arc]  (6) edge (1);

\end{scope}
\end{tikzpicture}

\caption{The eight oriented graphs in $F_6$.}
\label{Fig:F6}
\end{center}
\end{figure}

\begin{corollary}\label{C5hom}
The following statements are equivalent for a
graph $G$.
\begin{itemize}
	\item $G$ is homomorphic to the $5$-cycle,

	\item $G$ admits an orientation that has no
    semi-walk with pattern $\to \to \leftarrow
    \to \to$,

	\item $G$ admits an $\{\overrightarrow{C}_3,
    TT_3, \overrightarrow{P}_4, \overrightarrow{C}_4,
    C'_4, Q_6, D_5, \overleftarrow{D}_5\}$-free
    orientation, and

	\item $G$ admits an acyclic $\{TT_3,\overrightarrow{P}_4,
		C'_4, Q_6, D_5, \overleftarrow{D}_5\}$-free
    orientation.
\end{itemize}
\end{corollary}

\section{Conclusions}
\label{sec:Conclusion}

Consider an odd integer $n$, $n \ge 5$, and an oriented
graph $G$. Note that by the recursive definition of
$Cyc(n,G)$ and the proof of Theorem~\ref{iffCyc}, we
obtain a polynomial-time certifying algorithm that
determines if an oriented graph $G$ is homomorphic to
$AC_n$.  The yes-certificate is the cover $Cyc(n,G)$
that induces a homomorphism $\varphi \colon G \to AC_n$,
and the no-certificate is a semi-walk $W$ with $l$ arcs
of $G$, such that $l$ is even, $4\le l\leq n+1$, and $W$
follows the same pattern as $Q_l$.

As a nice consequence, we obtain that graphs admitting
a homomorphism to an odd cycle can be characterized as
those graphs having an orientation avoiding a well
defined finite set of oriented graphs.

Theorem \ref{tree-dual} seems to suggest that the
existence of a dual for an oriented tree, which is
linear on the order of the tree, is not such a rare
phenomenon.   Although the evidence for an affirmative
answer is sparse, we finish this work by proposing the
following questions.

\begin{question} \label{qst:longshot}
Is it true that for any oriented tree $T$, there is a
dual $D_T$ of $T$ of linear size with respect to
$|V_T|$?
\end{question}

In the event that the answer to Question \ref{qst:longshot}
results negative, from the results obtained in the
present work, the following question still makes sense.

\begin{question}
Is it true that for any oriented path $P$, there is a
dual $D_P$ of $P$ of linear size with respect to
$|V_P|$?
\end{question}


\begin{thebibliography}{15}

\bibitem{bloomJGT11}
	G.~S.~Bloom and S.~A.~Burr,
	On unavoidable digraphs in orientations of graphs,
	Journal of Graph Theory 11(4) (1987) 453--462.

\bibitem{fulkersonPJM15}
	D.~R.~Fulkerson and O.~A.~Gross,
	Incidence matrices and interval graphs,
	Pacific J. Math. 15(3) (1965) 835--855.

\bibitem{gallaiPCT}
	T.~Gallai,
	On directed paths and circuits,
	Theory of Graphs (Proc. Colloq., Tihany, 1966),
  Academic Press, New York, 1968, 115--118

\bibitem{guzmanArXiv}
  S.~Guzm\'an-Pro and C.~Hern\'andez-Cruz,
  Orientations without forbidden patterns on three vertices,
  arXiv:2003.05605.

\bibitem{hasseIMN28}
	M.~Hasse,
	Zur algebraischen Begrundung der Graphentheorie,
	I, Math. Nachr. 28 (1964/1965) 275--290.

\bibitem{haggkvistC8}
	R.~H\"aggkvist, P.~Hell, D.~J.~Miller, and V.~Neumann-Lara,
	On multiplicative graphs and the product conjecture,
	Combinatorica 8 (1988) 63--74.

\bibitem{hell2004}
	P.~Hell, and J.~Ne\v{s}et\v{r}il,
	Graphs and Homomorphisms, volume 28 of Oxford Lecture
  Series in Mathematics and its Applications,
	Oxford University Press 2004.

\bibitem{hellEJC12}
	P.~Hell, and J.~Ne\v{s}et\v{r}il,
	Images of Rigid Digraphs,
	European Journal of Combinatorics 12(1) (1991) 33--42.

\bibitem{hellC13}
	P.~Hell, H.~Zhou, and X.~Zhu,
	Homomorphisms to Oriented Cycles,
	Combinatorica 13(4) (1993) 421--433.

\bibitem{hellJDM8}
	P.~Hell, X.~Zhu,
	The Existence of Homomorphisms to Oriented Cycles,
	SIAM Journal on Discrete Mathematics 8(2) (1995).


\bibitem{nesetrilEJC29}
	J.~Ne\v{s}et\v{r}il, and C.~Tardif,
	A dualistic approach to bounding the chromatic number of a graph,
	European Journal of Combinatorics 29 (2008) 254--260.

\bibitem{nesetrilJCTB80}
	J.~Ne\v{s}et\v{r}il, and C.~Tardif,
	Duality theorems for finite structures (characterizing
  gaps and good characterizations),
	J. Combin. Theory Ser. B 80 (2000) 80--97.

\bibitem{royIRO1}
	B.~Roy,
	Nombre chromatique et plus longs chemins
  d\textsc{\char13}un graphe,
	Rev. Fr. Inform. Rech. Oper. 1 (1967) 129--132.

\bibitem{skrienJGT6}
	D~.J.~Skrien,
	A relationship between triangulated graphs,
  comparability graphs, proper interval graphs, proper
  circular-arc graphs, and nested interval graphs,
	Journal of Graph Theory 6(3) (1982) 309--316

\bibitem{vitaverDAN147}
	L.~M.~Vitaver,
	Determination of minimal colouring of vertices of
  a graph by means of Boolean powers of the incidence matrix,
	Dokl. Akad. Nauk SSSR 147 (1962) 758--759 (in Russian).

\bibitem{zhuJA19}
	X.~Zhu,
	A Polynomial Algorithm for Homomorphisms to Oriented Cycles,
	Journal of Algorithms 19(3) (1995) 333--345.

\end{thebibliography}
\end{document}